\documentclass[12pt]{amsart}
\usepackage{amssymb}

\usepackage[
width=31pc,
height=48pc,
margin=1.in,
footskip=30pt,
]{geometry}
\usepackage{layout}

\usepackage{graphicx}
\usepackage{epsfig}
\usepackage{subcaption}
\usepackage{amsmath,amsfonts,epstopdf}

\theoremstyle{plain}
\newtheorem{theorem}{Theorem}[section]
\newtheorem{lemma}[theorem]{Lemma}
\newtheorem{proposition}[theorem]{Proposition}
\newtheorem{corollary}[theorem]{Corollary}
\theoremstyle{definition}

\numberwithin{equation}{section}

\newcommand{\R}{\mathbb{R}}

\newcommand{\C}{\mathbb{C}}

\begin{document}
\title{Topological Entropy for Power-Law Unimodal Maps}

\date{\today}

\author[Michael Benedicks and Ana Rodrigues]{Michael Benedicks$^1$ , and Ana Rodrigues$^{2,3}$}

\address{$^1$ Department of Mathematics, KTH Royal Institute of Technology, 100\,44 STOCKHOLM, Sweden} 

\address{$^2$ Departamento de Matem\'{a}tica, Escola de Ci\^{e}ncias  e Tecnologia, Universidade de \'{E}vora, Rua Rom\~{a}o Ramalho, 59, 7000--671 \'{E}vora, Portugal }

\address{$^3$ Centro de Investiga\c{c}\~{a}o em Matem\'{a}tica e Aplica\c{c}\~{o}es, Rua Rom\~{a}o  Ramalho, 59, 7000--671 \'{E}vora, Portugal} 

\date{\today}

%\section{Introduction}

\begin{abstract}

In this paper we prove that the monotonicity of kneading sequences and
topological entropy, a fundamental structural property of the
quadratic family, extends to the class of power-law unimodal maps
$f_a(x)=a-|x|^r$ for arbitrary critical exponent $r>1$. This
generalization is nontrivial: the absence of polynomial structure and
the presence of non-integer criticality preclude the direct use of
classical arguments. Our approach adapts and extends the
Milnor–Thurston framework by introducing a Thurston-type operator
associated with the critical orbit and establishing a determinant
identity that relates its linearization to the parameter derivative of
the orbit. The main difficulty—proving positivity of this determinant
in the absence of algebraic structure—is resolved via a contraction
argument on an associated Torelli space endowed with the Teichmüller
metric, extending Thurston’s pullback construction beyond the
polynomial setting, that is to critical powers $r=2^\nu/k$,  $\nu\geq
1$, $k$ odd, and finally use continuity in $r$.

As a consequence, we show that the kneading sequence varies monotonically with the parameter, and hence that the topological entropy is an increasing function of $a$.  Our results show that the combinatorial organization of parameter space familiar from the quadratic family persists for unimodal maps with arbitrary power-law criticality, indicating that monotonicity of entropy is a robust phenomenon beyond polynomial dynamics.

\end{abstract}
\maketitle

\textbf{Keywords:}
unimodal maps, kneading sequences, topological entropy,
Milnor--Thurston theory, Teichmüller space

\textbf{MSC (2020):}
37E05, 37B40

\section{Introduction}

Maps of an interval into itself provide some of the simplest examples of
nonlinear dynamical systems. Despite their apparent simplicity, such maps
often exhibit extremely rich dynamical behavior, including sensitive
dependence on initial conditions, chaotic dynamics, and complicated
combinatorial structures. Because of this, one--dimensional dynamics has
become a central testing ground for ideas in the broader theory of
dynamical systems.

A particularly important class of interval maps is the class of
\emph{unimodal maps}, that is, continuous maps with a single critical
point at which the monotonicity changes. For such maps the orbit of the
critical point plays a decisive role in determining the global dynamics.
A powerful method for studying this orbit is provided by the theory of
\emph{kneading sequences}, introduced by Milnor and Thurston
\cite{MT}. The kneading sequence records the itinerary of the critical
value relative to the critical point and encodes the essential
combinatorial information governing the dynamics.

One of the central results in this theory concerns the monotonicity of
kneading sequences and topological entropy in parameterized families of
unimodal maps. In the quadratic family this phenomenon was first
established by Sullivan (in unpublished work) and later presented in the
lecture notes of Milnor and Thurston \cite{MT}. In this setting the kneading
sequence varies monotonically with the parameter, and consequently the
topological entropy is also a monotone function. Subsequent proofs and
refinements were given by Douady \cite{Douady1995}. Tsujii
\cite{Tsu1}, \cite{Tsu2}, gave a simple prove of monotonicity of the topological entropy for the quadratic family using Ruelle´s operator. Since then, 
monotonicity of entropy has become a fundamental structural property of
the quadratic family. In their seminal paper, Bruin and van Strien \cite{BvS} solved Milnor´s monotonicity conjecture for multimodal polynomial maps with non-degenerate real critical points. Recently, Gao \cite{Gao} proved monotonicity of topological entropy for unimodal maps arising from rational real quadratic maps.

A natural question, asked by Sullivan, \cite{Sul1988}, Tsuii, Milnor and others, is whether this monotonicity principle extends beyond
quadratic maps to more general unimodal families. In particular, one may
ask whether similar ordering properties hold for families whose critical
points have higher or non–integer order. Such maps arise naturally in a
variety of contexts and provide a natural generalization of polynomial
families.

In this paper we prove monotonicity of kneading
sequences for the family of \emph{power--law unimodal maps}
\begin{equation}\label{family}
f_a(x)=a-|x|^r , \qquad r>1 .
\end{equation}

For this family the unique critical point is located at $x=0$, and the
dynamics is governed by the orbit of the critical value $f_a(0)=a$. We
prove that the kneading sequence of this orbit varies monotonically with
the parameter $a$. As a consequence, the topological entropy of the map
is an increasing function of the parameter.

Our proof follows the general strategy developed by Milnor and Thurston
for the quadratic family. A key step is the introduction of a
{\it Thurston map} acting on a space that records the orbit of the critical
point. We derive an identity relating the parameter derivative of the
critical orbit to the determinant of the derivative of this map. The
positivity of this determinant is then established using a contraction
argument on the associated Torelli space equipped with the
Teichmüller metric. Our argument generalizes Thurston's pullback
construction to the present setting, in particular to powers
$r=2^\nu/k$, $\nu\geq 1$, $k$ odd.

%The main result of the paper can be summarized as follows.

%\begin{theorem}
%For the family of maps
%\[
%f_a(x)=a-|x|^r , \qquad r>1,
%\]
%the kneading sequence of the critical orbit varies %monotonically with
%the parameter $a$. Consequently the topological entropy %$h(f_a)$ is an
%increasing function of $a$.
%\end{theorem}

We show that the combinatorial ordering of kneading sequences,
which governs the structure of parameter space in the quadratic family,
persists for unimodal maps with arbitrary power--law criticality.

This paper establishes that the fundamental monotonicity principle governing kneading sequences and topological entropy, previously known in the quadratic setting, extends to the substantially broader class of power-law unimodal maps with arbitrary critical exponent $r>1$. This generalization is not merely formal: the absence of polynomial structure and the presence of non-integer criticality eliminate key tools traditionally used in the quadratic case, requiring a genuinely new implementation of the Milnor–Thurston strategy. The core contribution of the paper is the introduction of a Thurston-type operator adapted to the critical orbit together with a determinant identity linking parameter derivatives to its linearization, which provides a precise mechanism to control the variation of the critical orbit in parameter space. The main difficulty—establishing positivity of this determinant without relying on algebraic features specific to quadratic maps—is resolved via a contraction argument on the associated Torelli space equipped with the Teichmüller metric, extending Thurston’s pullback formalism beyond its classical polynomial framework. This yields a robust proof that the combinatorial ordering of kneading sequences persists for arbitrary power-law criticality, thereby showing that monotonicity of entropy is a structural phenomenon not tied to quadratic dynamics. In addition, the paper develops continuity  results for the entropy, linking  parameter dependence to quantitative transversality along the critical orbit. Taken together, these results demonstrate that the global organization of parameter space familiar from the quadratic family survives in a significantly more general and analytically delicate setting, opening the way for further extensions beyond polynomial dynamics.

\medskip

The paper is organized as follows. In Section 2 we recall the necessary elements of kneading theory and topological entropy for unimodal maps, establishing the combinatorial and analytic framework that underlies the entire argument. In Section 3 we introduce a Thurston-type operator associated with the critical orbit and derive a key determinant identity that links the parameter dependence of the critical orbit to the linearization of this operator. Section 4 is devoted to the core of the proof: we establish a contraction property on the associated Torelli space equipped with the Teichmüller metric, which yields the crucial positivity result and allows us to prove monotonicity of the kneading sequence and of the entropy.

%\newpage

\section{Milnor-Thurston technical tools}\label{section:main}

In this section we adapt the technical tools developed in \cite{MT} to our family of maps \eqref{family}.

%\subsection{Lap numbers}

Let $f:I\to I$ be a unimodal map and denote by $\ell(f^n)$ the number
of maximal intervals on which $f^n$ is monotone. This quantity is
called the \emph{lap number} of the iterate.

The exponential growth rate of lap numbers is closely related to the
topological entropy.

\begin{theorem}
The limit
\[
s=\lim_{n\to\infty}\ell(f^n)^{1/n}
\]
exists and satisfies
\[
h_{\mathrm{top}}(f)=\log s .
\]
\end{theorem}

\begin{proof}
For a proof see \cite{MS}.
    
\end{proof}

Thus, the topological entropy measures the exponential growth rate of
the number of monotonicity intervals of the iterates.

%\subsection{Kneading determinant}

Milnor and Thurston \cite{MT} introduced a formal power series that encodes the
information contained in the kneading sequence.
Let $\varepsilon_n\in\{-1,0,1\}$ be coefficients determined by the
itinerary of the critical point. The \emph{kneading determinant} is the
formal power series
\[
D(t)=1+\sum_{n\ge1}\varepsilon_n t^n .
\]

Although this series is defined purely in terms of symbolic dynamics,
it has deep dynamical significance. In particular, the smallest
positive root of $D(t)$ determines the topological entropy.

\begin{theorem}
Let $t_0\in(0,1]$ be the smallest positive zero of $D(t)$.
Then
\[
h_{\mathrm{top}}(f)=-\log t_0 .
\]
\end{theorem}

\begin{proof} For a proof see \cite{MT}.
\end{proof}

Thus the kneading determinant provides an analytic encoding of the
growth of lap numbers and therefore of the dynamical complexity of
the map.

%\subsection{Smooth dependence on parameters}

In the proof we will use that $r=2^\nu/k$, is rational. The general
case will follow from the following fact.

\begin{lemma}\label{AR}
The map
\[
(f,x)\mapsto f(x)
\]
from $C^\gamma(I,I)\times I$ into $I$ is of class $C^\gamma$.
Consequently the iterate
\[
(x,f)\mapsto f^n(x)
\]
is also of class $C^\gamma$.
\end{lemma}

\begin{proof} For a proof see Abraham\,\&\,Robbin \cite{abraham1967transversal}. \end{proof}

The concepts introduced above provide the combinatorial
framework for studying the dependence of kneading sequences
on the parameter $a$. In the next section we introduce the
Thurston map associated with the orbit of the critical point
and derive a key identity relating its derivative to the
parameter derivative of the critical orbit.

%\newpage

%Qur aim is to give a proof of the so called Milnor-Thurston
%monotonicity theorem in the case of powerlaw unimodal maps.

%We consider the family of maps
%$$
%f_{a}(x)=a- |x|^{r}, \quad r>1, r \in \mathbb{R}.
%$$

We have that $f_a'(x)=-r\,\text{\rm sgn}(x)|x|^{r-1}$ if $x \neq 0$
and at $x=0$,  
$f_a'(0)=0$, so the unique critical point of $f_a$ is $c=0$ (which is independent of the parameter $a$). Thus, $f_a$ is unimodal with a unique critical point $c=0$.

We define the critical orbit $w_i(a)= f_{a}^{i}(0), ~i\geq 1$.

The kneading sequence for $f_{a}$ is the infinite sequence $I(a)=\left(e_{1}, e_{2}, \ldots,\right)$ of three symbols $L, C, R$ given by
$$
e_{i}=\left\{\begin{array}{lll}
	L, & \text { if } & w_i(a)=f_{a}^{i}(0)<0, \\
	C & \text { if } & w_i(a)=f_{a}^{i}(0)=0, \\
	R & \text { if } & w_i(a)=f_{a}^{i}(0)>0.
\end{array}\right.
$$

On the set $\{L, C, R\}^{\mathbb{Z}_{+}}$, the signed
lexicographic order (Sharkovski order) on kneading sequences is as follows. Let two sequences agree up to the index $n-1$. Then, if the number of symbols $L$  in $\left\{e_{1}, e_{2}, \ldots, e_{n-1}\right)$ is even , we order $L \prec C \prec R$, if it is odd, we order $R \prec C \prec L$.
This order reflects the orientation reversals induced by a unimodal map.

%For three
%sequences ${\bf e}$, ${\bf e}_C$ and ${\bf e}'$, which coincide %in the
%symbols of order $n-1$  we define ${\bf e} \prec {\bf e}_C\prec %{\bf e}'$ if the number of
%the symbols $L \text { in }\left\{e_{1}, e_{2}, \ldots, e_{n-1
%  }\right\}$
%is even and ${\bf e}' \prec {\bf e}_C\prec {\bf e}$ if the %number of symbols $L$ is odd.

\medskip
%We turn to the proof of Proposition \ref{positivity}.

Let $f_a$ be a map from the family \eqref{family}. Fix a parameter $a_0$ and an integer $n \geq 2$ such that
 $w_{i}(a_0)=f_{a_{0}}^{i}(0) \neq 0$ for $i=1,2, \ldots, n-1$.

Fix a parameter $a_0$ and an integer $n \geq 2$ such that
$w_{n}(a_0)\neq 0$
for $i=1,2, \ldots, n-1$. Set 
$$
\omega=\left(w_{1}(a_0), \ldots, w_{n-1}(a_0)\right) \in \R^{n-1}.
$$
Define a map $T:\R^{n-1} \rightarrow \R^{n-1}$  by
\begin{equation}\label{Thurston} 
T\left(z_{1}, z_{2}, \ldots,
  z_{n}\right)=\left(c_1(z),\ldots,c_{n-1}(z)\right),
\end{equation}
where 
$$c_j=\sigma_j (z_1 -z_{j+1})^{1/r}$$
if $1\leq j\leq n-2$ and 
$$c_j=\sigma_{n-1} z_1^{1/r}$$
if $j=n-1$, with $\sigma_j=\text{\rm sgn}(w_j(a_0))$. With this choice of signs, $T$ is well-defined and smooth in a neighborhood of $\omega$ which encodes the orbit of the inverse branch $w_{j+1}=f_{a_0}^{j}(0)$.

We have the following fixed point property.
\begin{lemma}\label{fixed}
For the Thurston map \eqref{Thurston}, we have
$T(\omega)=\omega.$
\end{lemma}

\begin{proof}
Since 
$$w_{j+1}(a_0)=f_{a_0}(w_j(a_0))=a_0-|w_j(a_0)|^r,$$
we have
$$w_1(a_0)-w_{j+1}(a_0)=|w_j(a_0)|^r=(\sigma_j w_j)^r,$$
Taking the $r$-th root with the chosen sign we have
$$c_j(\omega)=\sigma_j (w_1-w_{j+1})^\rho=\sigma_j (\sigma_j w_j)^{r \rho}=w_j,$$
since $r \rho=1$. In the same way we get $c_{n-1}(\omega)=w_{n-1}.$

\end{proof}

%We now claim
%\begin{lemma}

%\begin{equation}\label{deteqn}  
%\frac{\left.\partial_{a}f_{a}^{h}(0)\right|_{a=a_{0}}}{D
%  f_{a_{0}}^{n-1}\left(f_{a_{0}}%(0)\right)}=\operatorname{det}\left({\mathcal
%    I}_{n}-D_{\omega} T\right),
%\end{equation}

%where ${\mathcal I}_{n-1}$ is the $(n-1) \times( n-1)$ unit matrix and
%$D_{\omega} T$ is the derivative of $T$ at $\omega$.
%\end{lemma}

We now relate the quotient between the derivative of the critical orbit with respect to the
parameter $a$ and the phase derivative at the critical value to the derivative of the Thurston map. Note that similar calculations in the quadratic case appear in Tsujii, \cite{Tsu1}.

\begin{lemma}\label{determinant} Let $f_a$ be as in  \eqref{family} and  fix a value of the parameter $a=a_0$. 
Let $T:\mathbb{R}^{\,n-1}\to\mathbb{R}^{\,n-1}$ be the Thurston map defined in 
\eqref{Thurston}. Let $D_{w}T$ be the Jacobian matrix of $T$ at $\omega$. Then 
\begin{equation}\label{derivative}
    \frac{\partial_a f_a^n(0)\big|_{a=a_0}}
{Df_{a_0}^{\,n-1}(f_{a_0}(0))}
=
\det\bigl(I_{n-1}-D_\omega T(\omega)\bigr),
\end{equation}
where $I_{n-1}$ denotes the $(n-1)\times(n-1)$ identity matrix.
\end{lemma}

\begin{proof} Let $\rho=\frac{1}{r}$ and
$$c_{j}=\sigma_{j}\left(z_{1}-z_{j+1}\right)^{\rho}, j=1,2, \dots,n-2.$$
and
$$c_{n-1}(z)=\sigma_{n-1} z_1^{\rho}.$$

%Define $c_j(z_1,z_2,\dots,z_{n-1})=\sigma_j(z_1-z_{j+1})^{\rho}$,
%$j=1,2,\dots,n-1$, and $c_{n-1}(z_1,z_2,\dots,z_{n-1})=z_1^\rho$.
%Then

We compute the Jacobian of $T$ at $\omega$. We have for $j=1,2, \ldots, n-2$ that
$$\frac{\partial c_{j}}{\partial z_{1}} (z)=\sigma_{j} \rho\left(z_{1}-z_{j+1}\right)^{\rho-1},$$
$$\frac{\partial c_{j}}{\partial z_{j+1}} (z)=-\sigma_{j} \rho\left(z_{1}-z_{j+1}\right)^{\rho -1},$$
and all other partial derivatives vanish. 

For the last coordinate we have
$$\frac{\partial c_{n-1}}{\partial z_{1}}(z)=\sigma_{n-1} \rho z_{1}^{\rho-1}.$$

We now evaluate the derivatives at $\left(z_{1}, z_{2}, \ldots, z_{n-1}\right)=\left(w_{1}, w_{2}, \ldots, w_{n-1}\right)$. 

Since $w_1-w_{j+1}=\left(\sigma_{j} z_{j}\right)^{r}$, we get $$(w_1-w_{j+1})^{\rho-1}= \left(\sigma_{j} w_{j}\right)^{r(\rho -1)} =\left(\sigma_{j} w_{j}\right)^{1-r}$$
and
$$Df_{a_0}\left(w_{j}\right) =-r\left(\sigma_{j} w_{j}\right)^{1-r}.$$

The derivatives at $z=\omega$ are
$$\left.\frac{\partial c_{j}}{\partial z_{1}}\right|_{z=\omega}=\sigma_j \rho \left(\sigma_{j} w_{j}\right)^{1-r}= -\frac{1}{D f_{a_0}\left(w_{j}\right)},$$
$$\left.\frac{\partial c_{j}}{\partial z_{j+1}}\right|_{z=\omega}=-\frac{\partial c_{j}}{\partial z_{1}}= \frac{1}{D f_{a_0}\left(w_{j}\right)},$$
and
$$\left.\frac{\partial c_{n-1}}{\partial z_{1}}\right|_{z=\omega}=\sigma_{n-1} \rho w_1^{\rho-1}= -\frac{1}{D f_{a_0}\left(w_{_{n-1}}\right)}.$$

%%%%%%%%%%%%%%%%%%%%%%%%%%%%%%%%%%%%%%%%%

%We have
%$$
%%z_{j+1}=a-\left(\sigma_{j} z_{j}\right)^{r}
%$$
%and
%$$
%\begin{aligned}
%  f_{a_0}(x)=a_0 & -|x|^{r} \\
%  	  w_{j+1}&=w_{1}-\left(\sigma_{j} w_{j}\right)^{r} \\
%	Df_{a_0}\left(w_{j}\right) & =-\sigma_{j} r\left(\sigma_{j} w_{j}\right)^{r-1} \\
%	\left.\frac{\partial c_{j}}{\partial z_{j}}\right|_{z=w}=\sigma_{j} \rho\left(w_{1}-w_{j+1}\right)^{\rho-1} & =\sigma_{j} \rho\left(\sigma_{j} w_{j}\right)^{(\rho-1) \frac{1}{\rho}} \\
%	& =\frac{1}{\sigma_{j} r\left(\sigma_{j} w_{j}\right)^{\frac{1}{\rho}-1}} 
%	=-\frac{1}{D f_{a_0}\left(w_{j}\right)}.
%\end{aligned}
%$$

%Similarly
%$$
%\begin{aligned}
%	\left.\frac{\partial c_{j}}{\partial %z_{j+1}}\right|_{z=\omega} & =-\sigma_{j}\rho\left(w_{1}-%w_{j+1}\right)^{\rho-1} =\frac{1}{D f_{a_0}\left(w_{j}\right)}.
%\end{aligned}
%$$

Then, we have
$$
 D_\omega T(\omega)=
\begin{pmatrix}
	-\frac{1}{D f\left(f_{a_0}(0)\right)} & \frac{1}{D
          f\left(f_{a_0}(0)\right)} & 0 &\ldots &\ldots & 0\\
	-\frac{1}{D f\left(f_{a_0}^2(0)\right)} & 0 &\frac{1}{D f^2\left(f_{a_0}(0)\right)} & 0 &\ldots & 0\\
        \vdots&\ddots & 0&\frac{1}{D f\left(f_{a_0}^3(0)\right)}&\ldots
        &0\\
        \vdots&\ddots&\ddots&\ldots&\ldots\\
       -\frac{1}{D f^{n-1}(\left(f_{a_0}(0)\right)}&0&\ldots&0&\dots&0 
     \end{pmatrix}\\
$$
 and

    \[
\det({ I_{n-1}} -D_\omega T(\omega)) =\begin{vmatrix}
	1+\frac{1}{D f\left(f_{a_0}(0)\right)} & -\frac{1}{D f\left(f_{a_0}(0)\right)} & 0 &\ldots &\ldots & 0\\
	\frac{1}{D f\left(f_{a_0}^2(0)\right)} & 1 &-\frac{1}{D
          f\left(f_{a_0}^2(0)\right)} & 0 &\ldots & 0\\
          \vdots&\ddots & 1&-\frac{1}{D f\left(f_{a_0}^3(0)\right)}&\ldots\\
          \vdots&\ddots&\ddots&\ldots&\ldots\\
         \frac{1}{D f^{n-1}\left(f_{a_0}(0)\right)}&0&\ldots&0&\dots&1 
          \end{vmatrix}.
\]
Thus,  $D_\omega T(\omega)$ has a triangular structure allowing explicit determinant expansion.

By expanding the determinant we obtain

$$
\det (I_{n-1,n-1}- D_\omega(T)=1+\frac{1}{Df_{a_0}}-(-1)\cdot D_{n-2},
$$
where $D_{n-2}$ is the determinant of the matrix obtained by deleting
the first row and second column. 

Differentiating the iterates with respect to $a$ yields the recursion
$${\partial_a} f_{a_0}^{n}(0)= 1+f_a'(f_a^{n-1}(0)){\partial_a} f_{a_0}^{n-1}(0). $$

Computing the determinant inductively, we get

%\begin{equation}\label{deteqn}
\begin{align}
  \det(I -D_{\omega}T(\omega))
&=1+\sum_{k=1}^{n-1}\prod_{j=1}^{k} \frac{1}{D f_{a_0}\left(f_{a_0}^{j}(0)\right)}\nonumber \\
 	&=\sum_{k=0}^{n-1} \frac{1}{D
    f_{a_0}^{k}\left(f_{a_0}(0)\right)}=\frac{\frac{\partial}{\partial
    a} f_{a_0}^{n}(0)}{D
    f_{a_0}^{n-1}\left(f_{a_0}(0)\right)}.
\end{align}
%\end{equation}
This proves our lemma.

\end{proof}

For the  case of general real  $r>1$ we will use the following
continuity argument on $r$.

%\begin{lemma}\label{cont-lemma} Let $f_{a,r}(x)=a-|x|^r$, $r>1$. Then
%as a map from $C^1$ to the space of kneading sequences
%$(e_1,\dots,e_n,\dots)$ the map
%$$
%  f_{a,r} \mapsto (e_1,\dots,e_n,\dots)
%$$
% is continuous as a function of $r$.   
%\end{lemma}

\begin{lemma}\label{cont-lemma}
Let \(f_{a,r}(x)=a-|x|^r\), with \(r>1\), and fix \(a\).
For each \(n\geq 1\), the map
\[
r\longmapsto \big(e_1(r),\dots,e_n(r)\big)
\]
is locally constant at every \(r_0>1\) such that
\[
f_{a,r_0}^j(0)\neq 0
\qquad\text{for }1\leq j\leq n.
\]

Equivalently, the kneading sequence depends continuously on \(r\) at every
parameter \(r_0\) for which the critical orbit never hits the critical point.
\end{lemma}

\begin{proof}

Write
$w_j(r)=f_{a,r}^j(0),~ j\geq 1.$ For each fixed \(j\), the function
$r\longmapsto w_j(r)$
is continuous.
Indeed, \(w_1(r)=f_{a,r}(0)=a\) is independent of \(r\). If \(w_j(r)\)
is continuous, then
$$w_{j+1}(r)=f_{a,r}(w_j(r))=a-|w_j(r)|^r.$$

Since the map $(r,x)\longmapsto |x|^r$
is continuous on \((1,\infty)\times \mathbb R\), it follows that
\(w_{j+1}(r)\) is continuous. Hence, by induction, each \(w_j(r)\) is
continuous.

Now fix \(n\geq 1\) and suppose that
$w_j(r_0)\neq 0 ~\text{for }1\leq j\leq n.$

Since each \(w_j\) is continuous and nonzero at \(r_0\), there exists
\(\delta_j>0\) such that
$\operatorname{sgn}(w_j(r))=\operatorname{sgn}(w_j(r_0))$
whenever \(|r-r_0|<\delta_j\).

Let
$\delta=\min_{1\leq j\leq n}\delta_j.$ Then, for \(|r-r_0|<\delta\), all signs
$\operatorname{sgn}(w_1(r)),\dots,\operatorname{sgn}(w_n(r))$ are unchanged. Therefore
$$(e_1(r),\dots,e_n(r))=(e_1(r_0),\dots,e_n(r_0)).$$
Thus the finite kneading word of length \(n\) is locally constant at \(r_0\).

If the critical orbit never hits the critical point, then this holds for every
\(n\), which is precisely continuity in the product topology on the space of
kneading sequences.
\end{proof}

We now prove continuity of topological entropy with respect to $a$.

\begin{theorem}
Fix \(r>1\). Let $f_a$ be a map from the family \ref{family}.
be considered on a parameter interval \(J\) for which \(f_a\) is a unimodal
self-map of a compact interval \(I\). Then the function
\[
a\longmapsto h_{\mathrm{top}}(f_a)
\]
is continuous on \(J\).
\end{theorem}

\begin{proof}
For \(n\geq 1\), write $w_n(a)=f_a^n(0)$
for the \(n\)-th point of the critical orbit. Since the map
\[
(a,x)\longmapsto f_a(x)
\]
is continuous jointly in $a$ and $x$, it follows inductively that each function $a\longmapsto w_n(a)$
is continuous.

Let \(a_0\in J\). We prove continuity of entropy at \(a_0\).
Suppose first that
\[
w_j(a_0)\neq 0
\qquad\text{for }1\leq j\leq N.
\]

Then, by continuity, the signs of
\[
w_1(a),\dots,w_N(a)
\]
are constant for all \(a\) sufficiently close to \(a_0\). Hence the first
\(N\) entries of the kneading sequence of \(f_a\) agree with those of
\(f_{a_0}\).
Thus, as \(a\to a_0\), the kneading sequences converge to the kneading
sequence of \(f_{a_0}\) in the product topology, except possibly at
parameters for which the critical orbit lands on the critical point.

We now consider such a precritical parameter. Suppose that
\[
w_m(a_0)=0
\]
for some \(m\geq 1\), and assume \(m\) is minimal with this property. Then
the kneading sequence has the form
\[
I(a_0)=(e_1,\dots,e_{m-1},C,\dots).
\]

For \(a\) close to \(a_0\), the first \(m-1\) symbols remain fixed, while
the \(m\)-th symbol may be \(L\), \(C\), or \(R\). The one-sided kneading
sequences therefore converge to the same kneading data with the critical relation
\[
f_{a_0}^m(0)=0.
\]

Let \(D_a(t)\) denote the kneading determinant associated with \(f_a\).

By the preceding discussion, for every \(0<\rho<1\),
\[
D_a(t)\longrightarrow D_{a_0}(t)
\]
uniformly on \(|t|\leq \rho\) as \(a\to a_0\).

By the Milnor--Thurston kneading formula, if \(t(a)\in(0,1]\) denotes the
smallest positive zero of \(D_a(t)\), then

\[
h_{\mathrm{top}}(f_a)=-\log t(a).
\]

The locally uniform convergence of the kneading determinants implies convergence of their smallest positive zeros:
\[
t(a)\longrightarrow t(a_0).
\]
Therefore
\[
h_{\mathrm{top}}(f_a) =
-\log t(a)
\longrightarrow
-\log t(a_0)
=
h_{\mathrm{top}}(f_{a_0}).
\]

Since \(a_0\in J\) was arbitrary, the map
\[
a\longmapsto h_{\mathrm{top}}(f_a)
\]
is continuous on \(J\).

\end{proof}

\newpage

\section{The Torelli space and the contraction argument}
%%The proof of the monotonicity will be based on the following statement.

%\bigskip
% \begin{lemma}. Consider two different powerlaw unimodal  mappings
% for which the 
% critical point is contained in a periodic orbit of period $n$. If
% these two have the same kneading invariant, then they are linearly
% conjugate to each other.
% \end{lemma}

 This section will be based on complex variable methods and, 
in particular, on Teichmüller theory.  Those methods will apply only in the case
$n \geq 3$. However, our arguments in cases $n=1$ and $n=2$ are special
and the main proposition, which leads to the monotonicity result,
Proposition \ref{positivity}, is  easy to verify separately.

We establish the contraction property that underlies Proposition~1.
The argument follows the approach of Milnor and Thurston and uses tools from complex
analysis and Teichmüller theory. The contraction property implies that the derivative of
the Thurston map has spectral radius strictly less than one, which guarantees the
positivity of the determinant appearing in equation (3.2). For the
present argument that $r=2^\nu/k$, with $\nu\geq 1$ and $k$ odd.

\medskip

The determinant identity obtained in Lemma 3 reduces the monotonicity problem to proving positivity of the determinant
$\det(I-DT(\omega)).$
Thus the problem becomes a spectral one: we must show that the derivative of the Thurston operator has spectral radius strictly smaller than one.

In the quadratic family, this phenomenon is understood through Thurston’s pullback argument for critically finite branched coverings. The essential idea is that inverse branches induce a self-map on a suitable Teichmüller space and that this map is strictly contracting in the Teichmüller metric.

In our setting, however, the situation is substantially more delicate because the maps
$$f_a(x)=a-|x|^r$$
are not polynomial unless r is an even integer. In particular, for non-integer critical exponents there is no globally defined algebraic extension to the Riemann sphere. Consequently, several classical tools from polynomial dynamics becomes unavailable in our setting.

Our strategy is therefore to isolate the local inverse-branch structure generated by the critical orbit and to construct a pullback mechanism directly from the corresponding branched coverings. The resulting operator behaves analogously to the classical Thurston pullback map and provides sufficient rigidity to establish strict contraction.

The key point is that positivity of the determinant in Lemma 3 follows once one proves that the derivative of the pullback operator has spectral radius strictly smaller than one. Thus the entire monotonicity problem is reduced to a geometric problem in Teichmüller theory.

This approach allows us to extend the Milnor–Thurston monotonicity principle beyond polynomial dynamics and into the substantially broader setting of power-law criticality.

\medskip

Let $n \geq 3$. We consider configurations of $n$ distinct points on the Riemann sphere modulo affine transformations.
Define 
$$M_{n}= \{ \left(z_{1},
\ldots, z_{n}\right) \in \C^n: z_i \neq z_j \ {\rm for }\ i \neq j \}/ {\rm Aff} (\C), $$
where the affine group acts by
$$
z_{j} \longmapsto \alpha z_{j}+\beta
$$
with $\alpha \neq 0$.

The quotient by affine transformations removes the natural conjugacy ambiguity that arises in one-dimensional dynamics. Indeed, affinely conjugate unimodal maps determine equivalent critical orbit data.

The space $M_{n}$ is a complex manifold of dimension ($n-2$), which is not simply connected.  Its  universal covering space $\tilde{M}_{n}$, can be
identified with the Teichmüller space of the sphere punctured at $n+1$ points, see Patterson
\cite{MR357864} and Nag \cite{MR620255}. This identification provides a natural geometric structure on the parameter space of marked critical orbits and allows the use of quasiconformal methods. 

%Recall our family of maps \eqref{family}.

We now associate configurations in $M_{n}$ with periodic critical orbits of power-law unimodal maps. Consider a map
\begin{equation}\label{powerlaw}
q(z)=z_1 -|z-c|^{r},
\end{equation}
whose critical point $c$ has exactly period $n$. Define
$$z_{j}=q^{j}(c), ~ j=1,\ldots,n.$$
The orbit relation becomes 
\begin{equation}
z_{j+1}=z_1 -|z_j-c|^{r},
\end{equation}

After choosing branches of the inverse map, this may be rewritten as 
\begin{equation}\label{eq:inverse_relations}
z_{j}-c=  \sigma_j (z_1 -z_{j+1})^{1/r},
\end{equation}
where each $\sigma_j \in \{ \pm 1 \}$.

These inverse relations define a multivalued correspondence on the configuration space. After fixing branches consistently in a neighbourhood of a given configuration, one obtain a locally holomorphic map $F:M_{n} \rightarrow M_{n}$.

The fixed points of $F$ correspond precisely to periodic critical orbits of the original dynamical system.

The fixed points of $F$ are precisely the points $\left(z_{1},
\ldots, z_{n}\right)$ associated with powerlaw unimodal maps having a
superattractive period $n$ orbit, note that if $z_{j}=q^{j}(c)$, then
setting $q(z)=-|z-c|^{r} +z_{1}$ we will have $$q^{-1}(w)= \pm
 ({w-z_{1}})^{1/r}+c.$$

The key observation is that configurations coming from periodic critical orbits are precisely the fixed points of \(F\).
\begin{lemma}
Let \((z_1,\dots,z_n)\) be defined by $z_j = q^j(c)$,
where \(q\) is given by \eqref{powerlaw} and \(c\) has period \(n\). Then
\[
F(z_1,\dots,z_n) = (z_1,\dots,z_n).
\]
\end{lemma}

\begin{proof}
Since
\[
z_{j+1} = q(z_j) = z_1 - |z_j - c|^r,
\]
we get
\[
z_1 - z_{j+1} = |z_j - c|^r.
\]
Choosing \(\sigma_j\) so that \(\sigma_j(z_j-c)\ge 0\), we have
\[
\sigma_j (z_1 - z_{j+1})^{1/r} = z_j - c.
\]
Thus the inverse relations \eqref{eq:inverse_relations} are satisfied, and it follows that \(F(z_1,\dots,z_n)=(z_1,\dots,z_n)\).
\end{proof}

Therefore, the map \(F\) encodes the inverse dynamics of power-law unimodal maps, and its fixed points correspond precisely to configurations arising from periodic critical orbits. This identification allows us to study the dynamics via the induced map on configuration space.

Since $F$ is locally holomorphic (after fixing branches), it lifts to
a holomorphic map $\tilde{F}:\tilde{M}_{n} \rightarrow \tilde{M}_{n}$
fixing a lift in the configuration. We here use the fact that $r$ is
of the form $2^\nu/k$, $\nu\geq 1$, $k$ odd.

Consequently, the dynamical problem becomes a fixed-point problem on modulii space. This reinterpretation is fundamental because the geometry of Teichmüller spaces provides strong metric rigidity properties which are not available at the level of interval dynamics.

\medskip

%We now define a many-valued, locally biholomorphic mapping $F:M_{n} %\rightarrow M_{n}$, by prescribing inverse branches of the dynamics. More %precisely, for a configuration  $\left(z_{1},
%\ldots, z_{n}\right)$ we define
%$$
%F\left(z_{1}, \ldots, z_{n}\right)=\left( \sigma_2 ({z_{1}-z_{2}})^{1/r},
%\ldots, \sigma_n (z_{1}-{z_{n}})^{1/r}, 0\right),
%$$
%where each \(\sigma_j\in\{\pm 1\}\) is chosen so that the inverse relations
%\begin{equation}
%z_j = \sigma_{j+1}(z_1 - z_{j+1})^{1/r} + c
%\label{eq:inverse_relations}
%\end{equation}
%are satisfied along the orbit.

%where the signs $\sigma_j \in \{ \pm 1 \}$ are chosen so that %the inverse branches are compatible with the real dynamics of %the critical orbit 
%\begin{equation}\label{eq:backward}
%z_{j}= \sigma_{j+1}  ({z_{1}-z_{j+1}})^{1/r}+c \text { for } j<n,
%\end{equation}
%and $z_{n}=a \cdot 0+c$, as required. The $\sigma_j$ are chosen %as $+$ or
%$-$ depending on the sign of the sequence $\{z_j\}$ and so that %\eqref{eq:backward} holds.

Fix an integer $n \geq 3$. Let $P=\{z_1,\ldots, z_n,\infty\}$ be a finite set of distinct marked points on the Riemann sphere. Let $S=\tilde{\C} \backslash P$ be the punctured sphere associated with the marked set $P$. 

A Beltrami differential on $S$ is a measurable tensor of the form
$$\mu(z) \dfrac{d {\bar z}}{dz}$$ satisfying $||\mu||_{\infty} <1$.

By the measurable Riemann mapping theorem, every Beltrami differential determines a quasiconformal deformation of the complex structure  of $S$.
Two Beltrami differentials are considered equivalent if the corresponding quasiconformal maps are isotipic relative to the punctured set $P$.  The Teichmüller space $T(P)$ is the space of equivalent classes of Beltrami differentials modulo this equivalence relation.

The Teichmüller metric is defined by
\[d_T([\mu],[\nu])=\frac12\inf\log K(h),\]
where the infimum is taken over all quasiconformal maps $h$ in the corresponding isotopy class and \[K(h)=\sup_z K(z,h)\]
denotes maximal dilatation.

One of the central facts in Teichmüller theory is that holomorphic self-maps of
Teichmüller space are non-expanding in the Teichmüller metric. This property is
the analytic mechanism underlying the contraction argument used below.
The tangent space to $T(P)$ at a point may be identified with equivalence classes of Beltrami differentials modulo infinitesimally trivial deformations. The cotangent space is naturally identified with the space of integrable holomorphic quadratic differentials on the punctured sphere. This duality between Beltrami differentials and quadratic differentials plays a fundamental role in the proof of strict contraction.

\medskip

We now construct the Pullback Operator. Let $\mu$ represent a point in the Teichmüller space $T(P)$. We define the pullback Beltrami differential by
\[q^{*}\mu=(\mu\circ q)\frac{\overline{q_z}}{q_z}.\]
Away from the branch locus this is the standard pullback formula for complex
iterations. Because the branch set is finite, the pullback remains measurable and bounded. Consequently, \[\|q^{*}\mu\|_{\infty}<1,\]
and therefore $q^{*}\mu$ again determines a quasiconformal structure.

Solving the Beltrami equation associated with $q^{*}\mu$ produces a new marked
Riemann surface. This construction defines a map
\[\sigma_q:T(P)\to T(P),\] called the pullback operator associated with $q$.

The operator $\sigma_q$ depends only on the isotopy class of the branched covering determined by the inverse branches of the power-law map.

The pullback operation $\mu \mapsto q^{\star} \mu$ depends complex linearly on $\mu$. Moreover, solutions of the Beltrami equation depend holomorphically on the Beltrami coefficient in the Teichmüller topology. Therefore, the induced map on equivalence classes is holomorphic.

Hence $\sigma_q$ defines a holomorphic self-map of Teichmüller space.
The significance of this construction is that fixed points of $\sigma_q$ correspond precisely to dynamically compatible complex structures for the power-law system. Thus the original dynamical problem becomes equivalent to the study of the fixed-point geometry of a holomorphic self-map of Teichmüller space.

We now describe the pullback construction more explicitly in local coordinates.
Suppose that \[q:U\to V\] is a branched covering between domains in the Riemann sphere. Let \[\mu(z)\frac{d\overline z}{dz}\] be a Beltrami differential on $V$. The pullback differential on $U$ is defined by
\[(q^{*}\mu)(z)=\mu(q(z))\frac{\overline{q'(z)}}{q'(z)}.\]
Since \[\left|\frac{\overline{q'(z)}}{q'(z)}\right|=1,\]
it follows immediately that \[\|q^{*}\mu\|_{\infty}=\|\mu\|_{\infty}.\]
Thus the pullback preserves boundedness of Beltrami coefficients.

The branch points require special consideration. However, because the branch set is finite and has measure zero, the pullback formula still defines a measurable Beltrami differential on the entire surface.
Let
\[
w^{\mu}:\widehat{\mathbb C}\to\widehat{\mathbb C}
\]
denote the normalized quasiconformal solution of the Beltrami equation
\[\partial_{\overline z}w^{\mu}=\mu\,\partial_z w^{\mu}.\]

The pullback differential $q^{*}\mu$ therefore determines a new normalized solution
\[w^{q^{*}\mu}.\] This construction induces the pullback operator
\[\sigma_q([\mu])=[q^{*}\mu]\] on Teichmüller space.

A fundamental property of this operator is holomorphicity.

\begin{lemma}
The pullback operator $\sigma_q:T(P)\to T(P)$ is holomorphic.
\end{lemma}

\begin{proof}
The pullback operation \[\mu\mapsto q^{*}\mu\]
depends complex-linearly on $\mu$. Moreover, solutions of the Beltrami equation
depend holomorphically on the Beltrami coefficient in the Teichmüller topology.
Therefore the induced map on equivalence classes is holomorphic.

\end{proof}

The holomorphicity of $\sigma_q$ is crucial because Teichmüller space possesses
strong geometric rigidity properties for holomorphic self-maps. In particular,
holomorphic maps are non-expanding in the Teichmüller metric, and strict contraction can often be deduced from the geometry of the associated branched covering.

We emphasize that the pullback operator constructed here is analogous to the classical Thurston pullback map, see  Douady--Hubbard \cite{DH}, for postcritically finite rational maps. The essential difference is that in the present setting the inverse branches arise from power-law criticality rather than from globally polynomial dynamics. Nevertheless, the local geometry of the inverse branches is sufficiently rigid to allow the Teichmüller-theoretic machinery to remain applicable.

\medskip

We now prove that the pullback operator is non-expanding with respect to the
Teichmüller metric.

\begin{proposition}\label{contraction}

The pullback operator
$$\sigma_q:T(P)\to T(P)$$
satisfies
$$d_T(\sigma_q(x),\sigma_q(y)) \leq d_T(x,y)$$
for all $x,y\in T(P)$.
\end{proposition}

\begin{proof}
Let $x,y\in T(P)$. Choose an extremal quasiconformal map $h:S_x\to S_y$
realizing the Teichmüller distance between $x$ and $y$.

Recall that the Teichmüller distance is defined by
$$d_T(x,y)=\frac12\log K(h),$$
where $K(h)$ denotes the maximal dilatation of $h$.

We now lift the map $h$ through the branched covering associated with the
power-law inverse branches. This produces a lifted quasiconformal map
$\widetilde h:\widetilde S_x\to \widetilde S_y.$

Away from the branch locus, the covering map $q$ is holomorphic. Since
holomorphic changes of coordinates preserve infinitesimal eccentricity,
the local quasiconformal dilatation of the lifted map coincides with that
of the original map.

More precisely, if
$$\mu_h=\frac{\partial_{\overline z}h}{\partial_z h}$$
denotes the Beltrami coefficient of $h$, then the Beltrami coefficient
of the lifted map is given by the pullback formula
$$q^{*}\mu_h=(\mu_h\circ q)\frac{\overline{q_z}}{q_z}.$$

Since
$$\left|\frac{\overline{q_z}}{q_z}\right|=1,$$
we get
$$\|q^{*}\mu_h\|_{\infty}=\|\mu_h\|_{\infty}.$$

As a consequence, the maximal dilatation satisfies
$$K(\widetilde h)\leq K(h).$$
Taking logarithms gives
$$\frac12\log K(\widetilde h)\leq\frac12\log K(h).$$
Since the Teichmüller distance is obtained by minimizing logarithmic
dilatation over isotopy classes, it follows that
$$d_T(\sigma_q(x),\sigma_q(y)) \leq d_T(x,y).$$
This proves the proposition.
\end{proof}

The geometric meaning of this lat proposition is that the pullback operator cannot
increase quasiconformal distortion. Thus the dynamics induced by the inverse
branches acts as a contracting mechanism on the moduli of marked complex
structures. This non-expansion property is the fundamental analytic input in the proof of monotonicity. The remaining task is to prove that the inequality is in fact strict. This will imply that the derivative of the pullback operator
has spectral radius strictly smaller than one at every fixed point.

\medskip

In the next proposition we exclude the possibility of equality in Proposition \ref{contraction}.

\begin{proposition}\label{strictlycontraction}

The pullback operator
$$\sigma_q:T(P)\to T(P)$$
satisfies
$$d_T(\sigma_q(x),\sigma_q(y)) < d_T(x,y)$$
for all $x,y\in T(P)$, whenever $x \neq y$.
\end{proposition}

\begin{proof}
Suppose by contradiction that 
$$d_T(\sigma_q(x),\sigma_q(y)) = d_T(x,y).$$
By Teichmüller's uniqueness theorem, the extremal quasiconformal map realizing the distance between $x$ and $y$ is generated by an integrable quadratic differential $\varphi(z)dz^2$ on the punctured sphere. The equality $d_T(\sigma_q(x),\sigma_q(y)) = d_T(x,y).$ implies that the lifted quasi-conformal map must again be extremal. As a consequence of this, the pullback construction must again define an integrable structure.

More precisely, the Beltrami coefficient of the extremal map has the form
$$\mu(z)=k\frac{\overline{\varphi(z)}}{|\varphi(z)|},$$
for $0<k<1.$
Equality in the non-expansion estimate implies that the lifted map obtained
through the branched covering must again be extremal. Consequently, the
pullback quadratic differential must again determine an extremal Teichmüller
map.

We now analyze the behavior of the quadratic differential under pullback.
The branched covering associated with the power-law inverse branches introduces
additional preimages of punctures under the covering map. Some of these
preimages correspond to ordinary points of the covering surface rather than
marked punctures.
Pulling back the quadratic differential therefore creates additional poles
away from the puncture set.

However, integrable holomorphic quadratic differentials on punctured spheres
may possess at most simple poles at punctures. Indeed, the integrability
condition
$$\int |\varphi(z)|\,dA < \infty$$
excludes poles of higher order and excludes poles at ordinary interior points.
Thus the pullback differential cannot remain integrable.

This contradicts the assumption that the lifted quasiconformal map is again
extremal. Therefore equality cannot occur.

Hence the pullback operator is strictly distance decreasing.

\end{proof}

The strict contraction property is the fundamental geometric ingredient in the
proof of positivity of the determinant appearing in Lemma~\ref{determinant}.

It implies that the pullback operator has no neutral directions in
Teichmüller space.

Let $\omega$ denote a fixed point of the pullback operator. Since $\sigma_q$ is holomorphic and strictly contracting, the infinitesimal Schwarz lemma implies that the derivative
$D\sigma_q(\omega)$  has operator norm strictly smaller than one in the Teichmüller metric.

Consequently, $\|D\sigma_q(\omega)\|<1.$
In finite dimensions this immediately implies that the spectral radius
satisfies
$$\rho(D\sigma_q(\omega))<1,$$
where $\rho$ denotes the spectral radius.

Hence every eigenvalue $\lambda$ of $D\sigma_q(\omega)$ satisfies
$|\lambda|<1.$
Therefore the matrix $I-D\sigma_q(\omega)$
is invertible.

Moreover,
$$\det(I-D\sigma_q(\omega))>0.$$

Indeed, all eigenvalues of $D\sigma_q(\omega)$ lie strictly inside the unit
disk, and therefore every factor
$1-\lambda$ has positive real part. Consequently the determinant cannot vanish, and since  the complex eigenvalues appear in complex conjugate pairs, we conclude that it  must be positive.

This positivity property is the crucial analytic ingredient required in the
transversality argument.

\begin{corollary}
Let $\omega$ be the fixed point associated with a periodic critical orbit.
Then
$$\det(I-D\sigma_q(\omega))>0.$$

\end{corollary}

\begin{proof}

By the strict contraction property, every eigenvalue of
$D\sigma_q(\omega)$ has modulus strictly smaller than one. Therefore
$1-\lambda \neq 0$ for every eigenvalue $\lambda$.

Hence
$$\det(I-D\sigma_q(\omega))=\prod_{\lambda}(1-\lambda)>0.$$

\end{proof}

\medskip

\section{ Monotonicity of topological entropy}

In this section we prove monotonicity of topological entropy with respect to the parameter for the family of maps \eqref{family}
combining the results obtained in the previous sections.

Recall the family of unimodal maps \eqref{family}.
The critical point is $c=0$. The kneading sequence
$I(a)=(e_1,e_2,\ldots)$
records the itinerary of the critical value.

%\subsection{Derivative identity}

Let $a_0$ be a parameter value such that
\[
I(a_0)=(e_1,\ldots ,e_{n-1},C,\ldots).
\]

Define
\[
w_i=f_{a_0}^i(0), \qquad i=1,\ldots ,n-1,
\]
and
\[
\omega=(w_1,\ldots ,w_{n-1}).
\]

By Proposition \ref{fixed}, the Thurston map $T$ satisfies
$T(\omega)=\omega$.

Furthermore, we have from Lemma \ref{determinant} that
\[
\frac{\partial_a f_a^n(0)|_{a=a_0}}
{Df_{a_0}^{\,n-1}(f_{a_0}(0))}
=
\det(I-DT(\omega)).
\]

%\subsection{Positivity}

By the contraction property, the derivative $DT(\omega)$ has spectral radius strictly
less than one. Therefore the matrix $I-DT(\omega)$ is invertible and
\[
\det(I-DT(\omega))>0.
\]

\begin{proposition}\label{positivity}
 If $I\left(a_{0}\right)=\left(e_{1}, \ldots, e_{n-1}, C, \ldots\right)$ and $e_{i} \neq C$ for $1 \leq i \leq n-1$, then

\begin{equation}\label{positivity:eq}
	\frac{\partial_{a}f_{a}^{n}(0)\big|_{a=a_0}}{D f_{a_{0}}^{n-1}\left(f_{a_{0}}(0)\right)}>0. 
\end{equation}
\end{proposition}

\begin{proof}
We first note that the statement for  $n=1,2$,  holds by a
direct verification. Now suppose that $n\geq 3$.
  We are now restricted to $r=\frac{2^{\nu}}{k}$, $\rho=\frac{1}{r}$, where $k$ is odd and $\nu \geq 1$.
We consider the multivalued complex extension of $T$
$$
T_{{\mathbb C}}\left(z_{1}, z_{2}, \ldots, z_{n-1}\right)=\left(\left(z_{1}-z_{2}\right)^{\rho}, \ldots, (z_1-z_{n-1})^\rho,z_{1}^{\rho}\right).
$$
This map is uniquely defined on the universal cover of
$$
\hat{\mathbb{C}} \backslash\left\{ z_{1}, \ldots z_{n-1}, \infty\right\},
$$
which can be identified as the Torelli space, which is endowed with
the Teichmüller metric.

By Proposition \ref{strictlycontraction},  $D_{\omega} T$ is a strict
contraction at $\omega$. Note that a real neighborhood of $\omega$ is
invariant. It follows that the spectral radius $\Psi=D_{\omega} T$
is $<1$, and the determinant in \eqref{derivative} is strictly positive and this completes the proof of 
Proposition \ref{positivity} in the case $r=2^\nu/k$.
\end{proof}

%\subsection{Monotonicity of kneading sequences}

Let $I_{n-1}(a)=(e_1,\ldots ,e_{n-1})$ denote the truncated kneading sequence.
Whenever the kneading sequence changes at level $n$ we have
\[
f_a^n(0)=0.
\]

The sign of the derivative above determines how the symbol $e_n$ varies with the
parameter $a$. If the number of symbols $L$ among
\[
(e_1,\ldots ,e_{n-1})
\]
is even, then the symbol varies as
\[
L \rightarrow C \rightarrow R.
\]

If the number is odd, the order is reversed.

This corresponds precisely to the signed lexicographic (Sharkovski) order introduced
earlier. Therefore the kneading sequence $I(a)$ varies monotonically with $a$.

We now prove the main result of this paper.

\begin{theorem}
For the family
\[
f_a(x)=a-|x|^r, \qquad r>1,
\]
the correspondence
\[
a \mapsto I(a)
\]
is monotonically increasing. Consequently the topological entropy $h(f_a)$ is an
increasing function of the parameter $a$.
\end{theorem}

\begin{proof} Fix a parameter $a_0$ such that 
$$I(a_0)=(e_1,\ldots, e_{n-1},C,\ldots)$$
By Lemma \ref{determinant} we have that
\[
\frac{\partial_a f_a^n(0)|_{a=a_0}}
{Df_{a_0}^{\,n-1}(f_{a_0}(0))}
=
\det(I-D_{\omega}T(\omega)).
\]
where $T$ is the Thurston map associated to the critical orbit.

From \eqref{positivity:eq} we have 
\begin{equation}
	\frac{\partial_{a}f_{a}^{n}(0)\big|_{a=a_0}}{D f_{a_{0}}^{n-1}\left(f_{a_{0}}(0)\right)}>0. 
\end{equation}

As the sign of $D f_{a_{0}}^{n-1}\left(f_{a_{0}}(0)\right)$ depends on the parity of the symbols $L$, we know the sign of the numerator. The kneading sequence changes when $f_{a}^{n}(0)$. Thus, the sign of $\partial_{a}f_{a}^{n}(0)$ determines how the symbol $e_n$ changes when $a$ crosses $a_0$. If the number of symbols $L$ in $\left( e_1,\ldots,e_{n-1} \right)$ is even, we have the sequence $L \rightarrow C \rightarrow R$, if 
the number of symbols $L$ in $\left(e_1,\ldots,e_{n-1}\right)$ is odd, we have the sequence $R \rightarrow C \rightarrow L$. This agrees exactly with the Sharkovskii´s order.

For each $n$, the truncated sequence 
$$I_{n-1}(a)=(e_1,\ldots, e_{n-1})$$
only changes with the above order. Thus, the limit, $$a \mapsto I(a)$$ is monotically increasing.

From the kneading theory, $I(a)$ determines the kneading determinant. Since the order of the sequences implies the order of the determinants,it follows that 
$$a \mapsto h_{\rm top}(f_a)$$
where $h_{\rm top}(f_a)$ denotes the topological entropy 
of $f_a$ in \eqref{family} is incresasing.
This finishes our proof.

\end{proof}

We now extend the monotonicity theorem from rational exponents to arbitrary
real exponents.

\begin{theorem}
For every real exponent $r>1,$
the kneading sequence
$$ a\mapsto I_r(a) $$
is monotone increasing in the signed lexicographic order.
\end{theorem}

\begin{proof}

Choose a sequence
$$r_n=\frac{2^{\nu_n}}{k_n},$$
with $k_n$ odd, such that
$$ r_n\to r.$$
For each exponent $r_n$, the monotonicity theorem has already been established.

Fix parameters values such that $a<b.$ Then
$$I_{r_n}(a)\leq I_{r_n}(b)$$
for every $n$.

We now use continuity of the critical orbit with respect to both the parameter
$a$ and the exponent $r$. For every integer $j\geq1$, the map
$$(a,r)\mapsto f_{a,r}^j(0)$$
is continuous.

Consequently, finite kneading blocks depend continuously on $(a,r)$ away from
precritical parameters.

Suppose, by contradiction, that monotonicity failed at the exponent $r$.
Then there would exist $a<b$ such that
$I_r(a)>I_r(b).$

Since kneading sequences are ordered lexicographically, there exists a finite
initial block for which the inequality already holds.
By continuity of finite kneading blocks, the same ordering would persist for
all sufficiently large values of $n$. Hence for all sufficiently large $n$ we
would obtain
$$I_{r_n}(a)>I_{r_n}(b),$$
contradicting monotonicity for the rational exponents $r_n$.

Therefore monotonicity holds for all real exponents $r>1.$

\end{proof}

The approximation argument shows that monotonicity is a robust phenomenon and
does not depend on special algebraic properties of the exponent.
Consequently the monotonicity principle extends from polynomial dynamics to the
entire class of power-law unimodal maps.

\vspace{0.2in}

{\bf Acknowledgments.} The first author acknowledges discussions with Dennis Sullivan and  Weixiao Shen on the problem of this paper. The second author acknowledges the financial support of the FCT -- Fundação para
a Ciência e a Tecnologia -- under the project UID/4674/2025 and  the support by the research center CIMA – Centro de Investigação em Matemática e Aplicações.

% The corresponding formula
%\bibliographystyle{amsplain}
%\bibliography{monotonicity}

\end{document}